\documentclass[10pt,leqno]{amsart}
\usepackage{tikz, subfigure, xcolor} 
\usepackage{cite}
\usepackage{hyperref}
\usepackage{comment}
\usepackage{inputenc}
\usepackage{enumerate}
\usepackage[normalem]{ulem}
\usepackage{doi}
\usepackage{dsfont}

\newtheorem{theorem}{Theorem}[section]
\newtheorem{lemma}[theorem]{Lemma}
\newtheorem{proposition}[theorem]{Proposition}
\newtheorem{corollary}[theorem]{Corollary}
\theoremstyle{definition}

\newtheorem{remark}[theorem]{Remark}

\numberwithin{equation}{section}

\providecommand{\norm}[1]{\left\| #1 \right\|} 

\renewcommand{\d}{\operatorname{d}\!} 
\DeclareMathOperator{\dt}{\d t}		
\DeclareMathOperator{\ds}{\d s}		
\DeclareMathOperator{\dW}{\d W} 	
\DeclareMathOperator*{\esssup}{ess\,sup}

\newcommand{\R}{\mathbb{R}}
\newcommand{\C}{\mathbb{C}}

\newcommand{\N}{\mathbb{N}}

\newcommand{\E}{\mathbb{E}}

\newcommand{\cA}{{\mathcal A}}

\newcommand{\cD}{{\mathcal D}}

\newcommand{\cH}{\mathcal{H}}
\newcommand{\cB}{{\mathcal B}}

\newcommand{\aA}{{\mathbb A}}

\title[Stochastic bidomain equations]{Global Strong Well-Posedness of the stochastic bidomain equations with FitzHugh--Nagumo transport}



\author[Hieber]{Matthias Hieber} 
\address{Department of Mathematics,
	TU Darmstadt, Schlossgartenstr. 7, 64289 Darmstadt, Germany}
\email{hieber@mathematik.tu-darmstadt.de}

\author[Hussein]{Amru Hussein}
\address{Department of Mathematics,
TU Kaiserslautern, Paul-Ehrlich-Stra{\ss}e,
67663 Kaiserslautern, Germany}
\email{hussein@mathematik.uni-kl.de}

\author[Saal]{Martin Saal}
\address{Scuola Normale Superiore, Piazza dei Cavalieri 7, 56126 Pisa, Italy}
\email{martin.saal@sns.it}

\begin{document}

\begin{abstract}
Consider the bidomain equations from electrophysiology with FitzHugh--Nagumo transport subject to current noise, i.e., subject to  stochastic forcing  modeled by a cylindrical Wiener process.
It is shown that this set of equations admits a unique global, strong pathwise solution within the setting of critical spaces. The proof is based on combining methods from stochastic and deterministic 
maximal regularity. In addition, the method of extrapolation spaces from deterministic evolution equations  is transferred to the stochastic setting.   
\end{abstract}

\maketitle


\section{Introduction}

The bidomain equations arise in various models describing the propagation of impulses in electrophysiology. These models have a long tradition, starting with the celebrated classical model by 
Hodgkin and Huxley in the 1950s. Following the descriptions in the monographs  by Keener and Sneyd ~\cite{KS98} and by Colli Franzone, Pavarino and Scacchi \cite{CPS14}, this system is given by
\begin{align}\tag{BDE}\label{BDE}
\left\{
\begin{aligned}
\partial_t u + f(u,w) - \nabla \cdot (a_i \nabla u_i) &= I_i & \mathrm{in} & ~ (0 , \infty) \times G,  \\
\partial_t u + f(u,w) + \nabla \cdot (a_e \nabla u_e) &= - I_e & \mathrm{in} & ~ (0 , \infty) \times G,  \\
\partial_t w + g(u,w) &= 0 & \mathrm{in} & ~ (0 , \infty) \times G,  \\
 u_i-u_e&= u & \mathrm{in} & ~ (0 , \infty) \times G  \\
\end{aligned}
\right.
\end{align}
subject to the boundary conditions
\begin{align}\label{eq:bc}
a_i \nabla u_i \cdot \nu = 0, \quad  a_e \nabla u_e \cdot \nu = 0  \quad \mathrm{on}  \quad  (0 , \infty) \times \partial G,
\end{align}
and the initial data
\begin{align}\label{BDESini}
u(0) = u_0, \quad w(0) = w_0  \quad \mathrm{in}  \quad G.
\end{align}
Here $G \subset \R^d$, $d=2,3$, denotes a domain, the functions $u_i$ and $u_e$ model the  intra- and extracellular electric potentials, $u$
 the transmembrane potential, and $\nu$ the outward unit normal vector to $\partial G$. The anisotropic properties of this system  are described by the conductivity matrices
$a_{i}(x)$ and $a_{e}(x)$. Furthermore, $I_i$ and $I_e$ stand for the intra- and extracellular stimulation currents, respectively. Concerning the ionic transport, we consider here the most 
classical model by FitzHugh--Nagumo, which reads
\begin{align*}
f(u,w)&= u(u-a)(u-1)+w = u^3-(a+1)u^2+au+w , \\
g(u,w)&= bw - cu,
\end{align*}
where $0<a<1$ and $b$, $c > 0$ are constants.

In this article we consider the {\em stochastic} bidomain equations subject to {\em current noise}, i.e., with a stochastic forcing term for the membrane potential modelled by a 
cylindrical Wiener process $W$ and a given function $h$.  This system is given by
\begin{align}\tag{S-BDE}\label{BDES}
\left\{
\begin{aligned}
\d u + f(u,w) dt - \nabla \cdot (a_i \nabla u_i)dt &= I_i \d t+ h\d W & \mathrm{in} & ~ (0 , \infty) \times G,  \\
\d u + f(u,w) dt + \nabla \cdot (a_e \nabla u_e)dt &= - I_e \d t + h\d W& \mathrm{in} & ~ (0 , \infty) \times G,  \\
\partial_t w + g(u,w) &= 0 & \mathrm{in} & ~ (0 , \infty) \times G,  \\
u_i-u_e&= u & \mathrm{in} & ~ (0 , \infty) \times G \\
\end{aligned}
\right.
\end{align}
subject to the above boundary condition \eqref{eq:bc} and the initial conditions \eqref{BDESini}.
Adding a stochastic term to the equations \eqref{BDES} for the membrane potential $u$ is usually called {\em current noise}, see \cite{GS11} for details. It represents the effect of random 
activity of ion channels on the voltage dynamics. In the case where $h$ depends on $u$, we arrive at the bidomain equations with {\em conductance noise}, see \cite{GS11}. We could, of course, add also 
Gaussian white noise to the ODE describing the evolution of the gating variable $w$ in \eqref{BDES} yielding the stochastic differential equation 
$$   
\d w = g(u,w) \d t  + h \d W.
$$  
This type of noise is called {\em subunit noise} in \cite{GS11}. In this article we concentrate however on the case of current noise, i.e., on equation \eqref{BDES}.

The rigorous mathematical analysis of the {\em deterministic} system was pioneered by  Colli Franzone and Savar\'e~\cite{CS00}, who introduced a variational 
formulation of the problem, showed global existence and uniqueness of weak  solutions in dimension $d=3$ for the FitzHugh-Nagumo ionic transport. Veneroni~\cite{Ven09} extended the latter result to 
more general models for the ionic fluxes. For optimal control results of the bidomain problem with various ionic transport laws we refer to \cite{KuWa14}.

Bourgault, Cordi\`ere, and Pierre presented in~\cite{BCP09} a new approach to this system by introducing for the first time the so-called bidomain operator within the $L^2$-setting. They
showed that it is a self-adjoint, positive, semi-definite operator, and proved existence and uniqueness of a local strong solution as well as the  existence of a global weak solution to the system, for
various classes of ionic models, including the one by FitzHugh--Nagumo. Giga and Kajiwara~\cite{GK16} recently gave a new stimulus to the investigation of this system by considering the bidomain 
equations within the $L^q$-setting for $q \in (1,\infty]$. They showed that the bidomain operator is the generator of an analytic semigroup on $L^q(G)$ for $q \in (1,\infty]$ and constructed a unique
local, strong solution to the bidomain system within this setting. 

The deterministic bidomain equations were studied recently also by Pr\"u{\ss} and the first author in \cite{HieberPruess} and \cite{HieberPruesslinear}. 
Using the theory of critical spaces, they proved, roughly speaking,  that the bidomain equations admit a {\em unique, global strong solution} in the two-dimensional setting for $u_0 \in L^2(G)$ and 
$w_0 \in L^q$ for $q \in [2,\infty)$ and in the three-dimensional situation for $u_0 \in H^{1/2,2}(G)$ and $w_0 \in L^q$ for  $q \in [2,\infty)$. Their approach was based on  rewriting  the bidomain 
equations as a semilinear evolution equation in  $X_0:= L^q(G)^2$ as
\begin{equation}\label{eq:bdabstract}
\partial_t v + Av = F(v),\quad t>0,\quad v(0)=v_0,  \quad A:=\left[\begin{array}{cc} \aA+a&1\\ -c&b\end{array}\right],  
F(v)=\left[\begin{array}{c}
-u^3+(a+1)u^2 \\0 \end{array}\right],
\end{equation}
where $v=[u,w]^{\sf T}$.

It is the aim of this article to show the existence of a {\em unique, global, strong} solution to the stochastic bidomain equation \eqref{BDES} in the pathwise sense for initial data 
belonging to certain critical spaces. 

Note that there are only very few results known on the deterministic or stochastic bidomain equation until today. This might be due to the fact, that the underlying 
bidomain operator $\aA$ is a highly {\em non local} operator. It is also interesting to compare the (deterministic or stochastic) bidomain system with the FitzHugh--Nagumo reaction diffusion system, where 
$\aA$ is replaced by a second order elliptic operator as the negative Laplacian $-\Delta$. In the latter case, there is a maximum principle, which yields-- by the method of invariant rectangles-- 
global existence of unique, strong solutions at least in the deterministic setting in all space dimensions. Since it is not known whether the bidomain operator $\aA$ has a maximum principal, we have 
to resort on different methods.          

Our approach can be described as follows: we first rewrite system  \eqref{BDES} with an additive noise as a semilinear stochastic evolution equation of the form   
\begin{align}\label{eq:bidomainstochaddnoise}
&\d U+ A U \dt =F(U) \dt + H \d W,
\end{align}
where $W$ is a cylindrical Wiener process, $A$ the non local operator defined in \eqref{eq:bdabstract} and, focusing on current noise (see \cite{GS11}), we assume that $H$ is of the form $H(t)=(h(t),0)^T$ for given $h$.
Secondly, we investigate the linearized system with linear noise 
\begin{align*}
&\d Z+ AZ \dt = H \d W,
\end{align*}
in the ground space $L^q(G)$ by the results on maximal stochastic regularity due to Van Neerven, Veraar and Weiss \cite{NeervenVeraarWeiss}. The latter are applicable due to the fact 
that $\aA$ admits a bounded $H^\infty$-calculus in $L^q(G)$, see \cite{HieberPruesslinear} or Proposition \ref{bidomopthm} below.  Thirdly, we consider  pathwise the remainder term $V:=U-Z$ for $Z$ which solves 
the system 
\begin{align}\label{eq:bidomainpathwiseaddnoise}
&\partial_t V+ AV =F(V+Z).
\end{align}
The maximal regularity properties of $Z$ allow us to  regard \eqref{eq:bidomainpathwiseaddnoise} as a deterministic, nonautonomous,  semilinear evolution equation. Then, extending the theory of 
critical spaces for semilinear equations developed originally by Pruess, Simonett and Wilke (cf. \cite{PruessSimonettWilke}) to the nonautonmous situation, we are able to 
prove the existence of a {\em global, strong} solution to \eqref{eq:bidomainstochaddnoise} for initial data belonging to critical spaces.  Observe that local existence results for smooth initial 
data could be achieved by standard arguments, however,  this is not the case for global existence results {\em without} smallness assumptions on the data.  The latter are related to \textit{a 
priori} estimates on the solution derived in Section~\ref{sec:global} and on estimates on the maximal existence interval $(0,T_{\max})$ of the local solution in certain {\em critical norms}. More 
precisely, we have 
\begin{align}\label{eq:global}
T_{\max}<\infty \Leftrightarrow \lim_{t\to T_{\max}} V(t) \text{ does not exist in } X_{\mu,p},
\end{align}
where $X_{\mu,p}$ denotes the interpolation space defined in Section 2.2. The fact that the bidomain operator $\aA$ admits a bounded $H^\infty$-calculus within the $L^q$-setting allows us to identify  
these interpolation spaces 
explicitly as Besov spaces.
Usual energy estimates are unfortunately {\em not} enough to relate the typical  energy norm estimates to these 
critical spaces and to apply \eqref{eq:global}. Our strategy is then to     
apply 
the theory of interpolation-extrapolation scales, cf. \cite[Section V.1]{Ama95}, to shift equation \eqref{eq:bidomainpathwiseaddnoise} from the ground space $X$ to suitable extrapolation spaces  
$X_{-1/2}$ or $X_{-1/4}$  of negative order, where depending on the space dimension the corresponding shifted interpolation spaces  $X_{\mu,p}$ can be related to the energy norms. 

 The \emph{stochastic} system \eqref{BDES} was investigated only very recently by Bendahmane and Karlsen \cite{BK19} within the context of martingale solutions. More precisely,
they established  the existence of a weak martingale solution to \eqref{BDES} for data in $L^2(\Omega,\cA,P;L^2(G))$ by means of an associated nondegenerate system and the Galerkin method. Moreover, they showed that 
equation \eqref{BDES} possesses a unique, weak solution provided the initial data  belong to $L^q(\Omega,\cA,P;L^2(G))$ for $q>9/2$. Whereas the results in \cite{BK19} can  be viewed within the PDE perspective as weak solutions lying in $H^{1,2}$, we are concerned with strong solutions to \eqref{eq:bidomainpathwiseaddnoise}
with 
\begin{align*}
V\in H^{1,p} ((\delta,T); L^q(G)) \cap L^p((\delta,T); H^{2,q}_N(G)) \times H^{1,p} ((\delta,T); L^q(G)), \quad \delta \in (0,T),
\end{align*}
for any $T>0$ and suitable parameters $p$ and $q$. We hence obtain a unique, global solution $(u,w)$ to the original equation \eqref{BDES} in the corresponding regularity class. Note that our 
approach using the bidomain operator circumvents the difficulties arising in the degenerate system treated in \cite{BK19}.

Abstract stochastic semi- and quasilinear evolution equation of the form
\begin{align}\label{eq:quasilinear}
	&\d u+ A(u) u \dt =F(u) \dt + H(u) \d W,
\end{align}
have been considered before by many authors, see e.g., \cite{FNS20, Roeckner, Cho14, Hornung, NeervenVeraarWeiss2, DaPZ92}. In fact, strong well-posedness 
results for \eqref{eq:quasilinear} were shown by 
van Neerven, Veraar and Weis \cite{NeervenVeraarWeiss2} as well as by Hornung \cite{Hornung} under Lipschitz conditions on $F$ and $H$.  Their results imply local existence results for  
\eqref{eq:bidomainstochaddnoise}. However, these results seem not to be applicable for obtaining global solution here. Our approach using the theory of critical spaces, allows us to apply the 
blow-up criteria for deterministic systems by relating critical spaces to energy norms. To this end, we need to shift our setting to suitable extrapolation spaces of negative order.

\section{Preliminaries}\label{sec:pre}
Throughout this article, $G \subset \R^d$ denotes a domain and $\Omega$ a probability space. For $1 \leq p \leq \infty$ and a Banach space $X$ let $L^p(G;X)$ be the Bochner space equipped with the norm
\begin{align*}
\|f\|_{L^p(G;X)}^p =\begin{cases}
\int_{G} \|f(x)\|^p_X \d x, & 1\leq p<\infty,\\
\esssup_{x\in G}\|f(x)\|_X, & p=\infty.
\end{cases}
\end{align*}
For $p=2$ and $X$ being a Hilbert space, the space $L^2(G;X)$ is a Hilbert space 
with scalar product $\left< f,g \right>=\int_{G} \left< f(x),g(x)\right>_X \d x $. 

\subsection{The Deterministic Bidomain Operator and Equation} \mbox{}\\
Here, we give a precise definition of the bidomain operator within the $L^q$-setting. To this end, let $G\subset\R^d$ be a bounded domain with 
boundary $\Gamma:=\partial G\in C^{2-}$.  We then define formally a pair of differential operators
$$\cA_k(x,D):= -{\rm div}( a_k(x)\nabla)=-\partial_i(a^{ij}_k(x)\partial_j),\quad x\in G,\quad k=1,2,$$
where the coefficient functions $a_k= (a^{ij}_k(\cdot))$ are given, and we employ the Einstein summation convention. Moreover, define the pair of boundary operators $\cB_k(D)$   by means of
$$
\cB_k(x,D) = \nu(x)\cdot a_k(x)\nabla = \nu_i(x) a_k^{ij}(x)\partial_j,\quad x\in \Gamma,\quad k=1,2,
$$
where $\nu(x)=[\nu_1(x), \ldots, \nu_d(x)]^T$ is the outer normal vector on $\Gamma$ at $x$.
Now, we introduce the following  assumptions on the coefficient functions:
\begin{align}\tag{\bf{BD}}\label{BD}
\begin{aligned}
(a) \quad & a_k \in W^{1,\infty}(G;\R^{d\times d}) \hbox{ are symmetric and uniformly positive definite on $\overline{G}$ for $k=1,2$.} \\
(b)\quad&  \hbox{There is a function $\gamma:\Gamma\to\R$ such that } \nu(x)\cdot a_2(x)=\gamma(x) \nu(x)\cdot a_1(x), \quad x\in\Gamma.
\end{aligned}
\end{align}
Note that condition (a) and (b) imply that for some $\gamma_0>0$
$$ \gamma(x) = \frac{\left<a_2(x)\nu(x),\nu(x)\right>}{\left<a_1(x)\nu(x),\nu(x)\right>}\geq \gamma_0>0,\quad x\in \Gamma,$$
and hence $\gamma\in W^{1,\infty}(\Gamma)$. 
It was observed in \cite[Remark 2.1 a)]{HieberPruesslinear} that condition b) is quite generic. From now on we always assume assumption (BD) to hold true. 

We proceed by introducing the spaces
$$
L_{0}^q(G):=\{ u\in L^q(G):\, \bar{u}=0\}, \quad \mbox{and} \quad H^{s,q}_{0}(G) = H^{s,q}(G)\cap L_{0}^q(G),\; s>0,
$$
where $\bar{u}:=\int_G u$ is the  mean value of $v$. Now, we define two operators $A_k$, $k=1,2$, in the base space $X_0:=L_{0}^q(G)$
by 
means of
\begin{equation}\label{def:Ak}
A_k u := \cA_k(\cdot,D) u,\quad \cD(A_k)=\{u\in H^{2,q}_{0}(G):\, \cB_k(\cdot,D) u =0 \mbox{ on } \Gamma\}.
\end{equation}
It is well-established that under condition (BD)(a) and (b), $A_k$ for $k=1,2$ is sectorial, boundedly invertible and admits an $\cH^\infty$-calculus with $\cH^\infty$-angle 0; see e.g. \cite{DDHPV}.
The conditions (BD) yields in particular 
\begin{align*}
\cD(A_1)=\cD(A_2)=:X_1 \quad \hbox{as well as} \quad \cD(A_k^\alpha)=(X_0,X_1)_\alpha \quad \hbox{for } k=1,2, \quad \alpha \in (0,1),
\end{align*}
where $(\cdot,\cdot)_\alpha$ denotes the complex interpolation functor.
The {\em bidomain operator} $\aA$ in $X_0$ is then defined as
\begin{equation}\label{bidomop}
\aA:= (A_1^{-1}+A_2^{-1})^{-1}, \quad \cD(\aA):=X_1.
\end{equation}
Note that $\aA= A_1(A_1+A_2)^{-1}A_2 = A_2(A_1+A_2)^{-1}A_1$. 

\begin{proposition}[Properties of the bidomain operator, cf. \cite{HieberPruesslinear}]\label{bidomopthm}
Let $1<p,q<\infty$, $X_0=L_{0}^q(\Omega)$, and assuming (BD) let the bidomain operator $\aA$ be defined as in \eqref{bidomop}. 
Then the following assertions are true: 
\begin{enumerate}[(a)]
	\item $\aA$ is sectorial and boundedly invertible in $X_0$.
	\item  $\aA$ admits a bounded $H^\infty$-calculus on $X_0$ of angle $0$, i.e. $\aA\in \cH^\infty(X_0)$ with $\phi^\infty_{\aA}=0$.
	\item The Cauchy problem associated with $\aA$  
		has maximal $L^p$-$L^q$-regularity on $\R_+$.
	\item For $z \in \C\setminus(0,\infty)$, the resolvent $(z-\aA)^{-1}$ of $\aA$ is a compact operator on $X_0$, and thus the  spectrum $\sigma(\aA)$ of $\aA$ consists only of eigenvalues with finite
	algebraic multiplicity.
	\item $-\aA$ generates a strongly continuous, compact, analytic and exponentially stable semigroup on $X_0$.
	\item $\cD(\aA^\alpha)=\cD(A_k^\alpha)=(X_0,X_1)_\alpha$ for $\alpha\in(0,1)$ and $k=1,2$.
\end{enumerate}
\end{proposition}

Extending $\aA$ trivially to $L^q(G)$  with domain $\cD(\aA)\oplus \hbox{span}\{1\}$, where with a slight abuse of notation the resulting operator is still denoted as bidomain operator $\aA$, the deterministic bidomain problem 
\eqref{BDE}
subject to the boundary conditions \eqref{eq:bc} and the initial data \eqref{BDESini}
can be reformulated as the system
\begin{align}\label{BDE1}
\begin{split}
\partial_t u+f(u,w) + \aA u &= 0,\quad t>0,\quad u(0)=u_0,\\
\partial_t w+g(u,w) &= 0,\quad t>0,\quad w(0)=w_0.
\end{split}
\end{align} 
The ionic transport is modeled by the classical FitzHugh--Nagumo equations, which are formulated as 
\begin{align*}
f(u,w)= u^3-(a+1)u^2+au+w \quad \hbox{and}\quad
g(u,w)= bw - cu,
\end{align*}
where $0<a<1$ and $b$, $c > 0$ are constants. Setting  $v:=[u,w]^{\sf T}$, $F(v):=[-u^3+(a+1)u^2,0]^{\sf T}$ and the matrix operator $A$ defined in the strong stetting in the base space $X_0^s:= L^q(G)^2$ by
\begin{equation}\label{def:A}
A:=\left[\begin{array}{cc} \aA+a&1\\ -c&b\end{array}\right], \quad \cD(A)= \cD(\aA) \times L^q(G)=:X_1^s,
\end{equation}
we obtain the formulation of \eqref{BDE1} as the semilinear evolution equation in $X_0^s$
\begin{equation}\label{sl-ev-eq}
\partial_t v + Av = F(v),\quad t>0,\quad v(0)=v_0.
\end{equation}
It should be noted that the classical FitzHugh-Nagumo system appears as a special case of this equation, assuming $u,w$ to be spatially constant.

\begin{remark}\label{rem:A}
In \cite{HieberPruess} it has been shown that the properties of $\aA$ given in Proposition~\ref{bidomopthm} except compactness of the resolvent carry over to $A$.	
\end{remark}
As a consequence of Remark~\ref{rem:A}, we see that the complex interpolation spaces and the fractional power domains of $A$ satisfy
\begin{equation}\label{eq:interpolation} 
\cD(A^\alpha)=(X_0^s,X_1^s)_\alpha = \cD((\aA)^\alpha)\times L^q(G),\quad \cD((\aA)^\alpha)=H^{2\alpha,q}_{N}(G),
\end{equation}
where the subscript $N$ indicates Neumann-type boundary conditions $\cB_k(\cdot,D)u=0$, which by \eqref{BD} is independent of $k=1,2$, 
whenever this trace exists.
More concretely, we have
\begin{align*}
H^{2\alpha,q}_{N}(G) =\begin{cases}
H^{2\alpha,q}(G), & 0<\alpha< 1/2+1/2q, \\
\{ u\in H^{2\alpha,q}(G)\colon \cB_k(\cdot,D)u=0 \mbox{ on }\Gamma\},  & 1/2+1/2q<\alpha<1.
\end{cases}
\end{align*}
The real interpolation spaces $D_A(\alpha,p)=(X_0^s,X_1^s)_{\alpha, p}$ for $p\in (1,\infty)$ and $\alpha \in (0,1)$ satisfy the relation
$$ D_A(\alpha,p) = B^{2\alpha}_{qp,N}(G)\times L^q(G),$$
where 
\begin{align*}
B^{2\alpha}_{qp,N}(G) =\begin{cases}
B^{2\alpha}_{qp}(G), & 0<\alpha< 1/2+1/2q, \\
\{ u\in B^{2\alpha}_{qp}(G)\colon \cB_k(\cdot,D)u=0 \mbox{ on }\Gamma\},  & 1/2+1/2q<\alpha<1.
\end{cases}
\end{align*}
Applying Amann's theory of interpolation-extrapolation scales \cite[Section V.1]{Ama95}, we see that the same properties of $\aA$ given in Proposition~\ref{bidomopthm}, and hence also those of $A$ by Remark~\ref{rem:A}, carry over 
to the spaces
\begin{align*}
X_0^{\sf w_1} &= H^{-1,q}(G)\times L^q(G),\quad X_1^{\sf w_1} =H^{1,q}(G)\times L^q(G),\quad  H^{-1,q}(G)=H^{1,q^\prime}(G)^*, \\
X_0^{\sf w_2} &= H^{-1/2,q}(G)\times L^q(G),\quad X_1^{\sf w_2}  =H^{3/2,q}_{N}(G)\times L^q(G),\quad  H^{-1/2,q}(G)=H^{1/2,q^\prime}(G)^*,
\end{align*}
where $\tfrac{1}{q}+\tfrac{1}{q^\prime}=1$.

\subsection{Nonautonomous Semilinear Parabolic Evolution Equations} \mbox{}\\
In contrast to the analysis of the deterministic bidomain equations as described in \cite{HieberPruess} and \cite{HieberPruesslinear}, the stochastic setting forces us to consider nonlinearities which 
are explicitly time dependent. To this end, we adapt the existence, uniqueness and stability results due to  Pr\"u{\ss}, Simonett and Wilke \cite{PruessSimonettWilke} 
to our nonautonomous situation. More precisely, for given functions $F_1,F_2$ consider the semilinear parabolic evolution equation
\begin{align}\label{eq:semilineardeterministic}
\begin{split}
\partial_t v + A v & =F_1(\cdot,v)+F_2(\cdot,v),\\
v(0)&=v_0
\end{split}
\end{align}
in  time weighted Sobolev spaces, which for $p \in (1,\infty)$, $\mu\in (1/p,1]$, a time interval $J \subset [0,\infty)$, and a Banach space $X$ 
are defined by  
\begin{align*}
L^{p}_{\mu}(J;X)&:=\{v \in L^1_{loc}(J;X): t^{1-\mu}u\in L^{p}(J;X)\} \mbox{ and } \\
H^{1,p}_{\mu}(J;X)&:=\{v\in L^{p}_{\mu}(J;X) \cap H_{loc}^{1,1}(J;X): t^{1-\mu}\partial_t v \in L^p(J;X)\}.
\end{align*}
Assume that $X_1,X_0$ be Banach spaces such that $X_1$ is densely embedded into $X_0$ and that $V_{\mu,p}$ is an open subset of the real interpolation space 
$$
X_{\mu,p}:=(X_0,X_1)_{\mu-1/p,p}, \quad \mu \in (1/p,1]. 
$$
The complex interpolation space is denoted by $X_{\beta}:=(X_0,X_1)_{\beta}$ for $\beta \in (0,1)$.
 For $1<p<\infty$ and  $1/p<\mu\leq 1$ we now introduce the following assumptions (A1)-(A3). 

\vspace{.1cm}\noindent
{\bf (A1):} Assume that 
\begin{align*}
L^{p}((0,T);X_{1}) \cap H^{1,p}((0,T);X_{0}) \hookrightarrow H^{1-\beta,p}((0,T);X_{\beta}), \quad T>0.
\end{align*}
Note, that this embedding holds true if there exists  a bounded operator $A: X_1\to X_0$ which admits a bounded $H^{\infty}$-calculus of angle strictly less than $\pi/2$.

\vspace{.1cm}\noindent
{\bf (A2):} Assume that $A: X_1 \to X_0$ is bounded and has maximal $L^p$-regularity.

\vspace{.1cm}\noindent
{\bf (A3):} For $V_{\mu,p}\subset X_{\mu,p}$ open and $T>0$  
\begin{align*}
F_1  \colon  [0,T] \times V_{\mu,p} \to X_0, \; F_2  \colon  [0,T] \times V_{\mu,p} \cap X_{\beta} \to X_0,
\end{align*}
satisfy
\begin{align*}
F_1(\cdot,v) \in L^{p}_{\mu}((0,T);X_0)  & \text{ for all }v\in C([0,T];V_{\mu,p}),\\
F_2(\cdot,v) \in  L^{p}_{\mu}((0,T);X_0) & \text{ for all }v\in C([0,T];V_{\mu,p})\cap H^{1-\beta,p}_{\mu}((0,T);X_{\beta}),
\end{align*}
\begin{align}\label{eq:estimatef1}
\|F_1(t,v_1)-F_1(t,v_2)\|_{X_0} \leq C\|v_1-v_2\|_{X_{\mu,p}} \quad \hbox{for } v_1,v_2\in V_{\mu,p},
\end{align}
and for $m\in\N$, $\rho_j\geq 0$, $\beta\in(\mu-1/p,1), \beta_j\in(\mu-1/p,\beta]$
\begin{align}\label{eq:estimatef2}
\|F_2(t,v_1)-F_2(t,v_2)\|_{X_0} \leq C \sum_{j=1}^{m} \left(1+\|v_1\|^{\rho_j}_{X_{\beta}}+\|v_2\|^{\rho_j}_{X_{\beta}} \right)\|v_1-v_2\|_{X_{\beta_j}}
\quad \hbox{for } v_1,v_2\in  V_{\mu,p}\cap X_{\beta},
\end{align}
where the constants $C$ are {\em independent} of $t\in [0,T]$ in both estimates 
and where for all $j\in \{0,\ldots,m\}$
\begin{align*}
\rho_j\beta+\beta_j-1\leq \rho_j(\mu-1/p).
\end{align*}

Note that our assumptions are essentially as in \cite{PruessSimonettWilke}, however,  we allow here $F_1$ and $F_2$ to be explicitly time dependent. The constants $C$ appearing in the   
estimates \eqref{eq:estimatef1} and \eqref{eq:estimatef2} may depend on $T$ and must be uniform in $t\in [0,T)$.

Modifying the proof of   \cite[Theorem 2.1]{PruessSimonettWilke}, we obtain  the following result on the existence and uniqueness of local solutions to equation \eqref{eq:semilineardeterministic}.

\begin{proposition}\label{theorem:localexistencedeterministic}
Let $1<p<\infty$, $1/p<\mu\leq 1$, $V_{\mu,p}\subset X_{\mu,p}$ open, $T>0$, and assume (A1)-(A3). Then for $v_0\in V_{\mu,p}$ there exists  $T_{0} \in (0, T]$ and a unique solution 
\begin{align*}
v\in H^{1,p}_{\mu}((0,T_{0});X_0)\cap L^{p}_{\mu}((0,T_{0});X_1) \cap C([0,T_{0}];V_{\mu,p}).
\end{align*}
The solution extends onto a maximal existence interval $[0,T_{\max})$.
Furthermore, there is an $\varepsilon>0$ and a constant $C>0$ with  $\overline{B(v_0,\varepsilon)}\subset V_{p}$ such that for all $w_0 \in \overline{B(v_0,\varepsilon)}$ there is a unique solution 
$w$ of \eqref{eq:semilineardeterministic} with the same regularity as $v$ and 
\begin{align*}
\|v-w\|_{H^{1,p}_{\mu}((0,T_{0});X_0)}+\|v-w\|_{L^{p}_{\mu}((0,T_{0});X_1)}+\|v-w\|_{C([0,T_{0}];V_{\mu,p}}\leq C \|v_0-w_0\|_{X_\mu,p}.
\end{align*}
\end{proposition}
The maximal existence time is characterized by
\begin{align*}
	T_{max}:=\sup\{T>0\colon v\in H^{1,p}_{\mu}((0,T);X_0)\cap L^{p}_{\mu}((0,T);X_1) \cap C([0,T];V_{\mu,p})	
	\hbox{ solves \eqref{eq:semilineardeterministic}} \}.
\end{align*}
If $\rho_j\beta+\beta_j-1 < \rho_j(\mu-1/p)$ we call $j$ \textit{subcritical}, in the case of equality we call it \textit{critical}, and  we define the \emph{critical weight} 
\begin{align*}
	\mu_c:= \tfrac{1}{p} + \beta - \min_{j} (1-\beta_j)/\rho_j.
\end{align*}	

\begin{remark}\label{rem:embedding}
	Note that for $T\in (0,\infty]$ one has the embedding
	\begin{align*}
	H^{1,p}_{\mu}((0,T);X_0)\cap L^{p}_{\mu}((0,T);X_1) \hookrightarrow BUC([0,T];X_{\mu,p}), 
	\end{align*}
	which allows one to take traces in time, 
	and  for any $\delta\in (0,T)$ one has the instantaneous smoothing
	\begin{align*}
	H^{1,p}_{\mu}((0,T);X_0)  \hookrightarrow
		H^{1,p}((\delta,T);X_0) \quad \hbox{and}\quad
		L^{p}_{\mu}((0,T);X_1) \hookrightarrow L^{p}((\delta,T);X_1).
	\end{align*}
\end{remark}

\begin{corollary}[cf. \cite{PruessSimonett} Corollary 5.1.2.]\label{corollary:continuation}
Let $V_{\mu,p}=X_{\mu,p}$ and $v$ be the solution of \eqref{eq:semilineardeterministic} given in Proposition \ref{theorem:localexistencedeterministic}. Then
If $T_{max}<\infty$, then 
one of the following alternatives occurs
\begin{itemize}
\item[i)] $\lim_{t\to T_{\max}} v(t)$ does not exist in $X_{\mu,p}$.
\item[ii)] $\lim \inf_{t\to T_{\max}} \mathrm{dist}(v(t),\partial V_{\mu,p})=0$.
\end{itemize}
	In the case $V_{\mu,p}=X_{\mu,p}$, this reduces to
\begin{align*}
T_{\max}<\infty \Leftrightarrow \lim_{t\to T_{\max}} v(t) \text{ does not exist in } X_{\mu,p}.
\end{align*}
\end{corollary}

The following  criteria for global existence is akin to the classical Serrin-type results for the Navier-Stokes equations.
\begin{proposition}[cf. \cite{PruessSimonettWilke} Theorem 2.4]\label{theorem:serrin}
Let $u$ be the solution of \eqref{eq:semilineardeterministic} given in Proposition~\ref{theorem:localexistencedeterministic} on its maximal interval of existence $[0,T_{\max})$ and the 
critical weight $\mu_c>1/p$. Then
\begin{itemize}
\item[i)] $u\in L^p((0,T);X_{\mu_c})$ for all $T<T_{\max}$.
\item[ii)] $u\notin L^p((0,T_{\max});X_{\mu_c})$ if $T_{\max}<\infty$.
\end{itemize}
\end{proposition}

\subsection{Stochastic maximal regularity} \mbox{}\\
In this subsection we put some of the results on stochastic maximal $L^q$-regularity developed in  \cite{NeervenVeraarWeiss} into the context of the linearized bidomain equation.  

Let $(\Omega, \mathcal{A},P)$ be a probability space with a filtration $\mathcal{F}=(\mathcal{F}_t)_t$. An 
$\mathcal{F}$-cylindrical Brownian motion on a Hilbert space $\cH$ is a bounded linear operator $\mathcal{W}: L^2((0,\infty);\cH)
\to L^2(\Omega)$ such that for all $f,g\in \cH, t'\geq t \geq 0$:
\begin{itemize}
\item[a)] The random variable $W(t)f:=\mathcal{W}(\mathds{1}_{[0,t]}\otimes f)$ is centered Gaussian and $\mathcal{F}_t$-measurable.
\item[b)] $\E[W(t')f\cdot W(t)g]= t \left< f,g \right>_\cH$.
\item[c)] The random variable $W(t')f-W(t)f$ is independent of $\mathcal{F}_t$.
\end{itemize}
If $\cH$ is separable and $(e_n)_n$ an orthonormal basis of $\cH$, then $\beta_n(t):=W(t)e_n$ is a standard $\mathcal{F}$-Brownian motion, and we have the representation 
\begin{align*}
W(t)f = \sum_{n=1}^\infty \beta_n(t) \left< f,e_n \right>_\cH.
\end{align*}
Hence, $W(t): \cH \to L^2(\Omega)$, $W(t)= \sum_{n=1}^\infty \beta_n(t) \left< \cdot ,e_n \right>_\cH$ defines a family of linear operators. Combining Proposition \ref{bidomopthm} b) with the results in 
\cite{NeervenVeraarWeiss} to the  linearized bidomain equation for given $H$, 
\begin{align}\label{eq:stoch}
\d Z(t) + A Z(t) \dt = H(t) \dW(t),
\end{align}
we obtain the following result on the stochastic convolution
\begin{align}\label{eq:stochconvolution}
Z(t):=\int_0^t e^{(t-s)A} H(s) \dW(s).
\end{align}

\begin{proposition}[cf. \cite{NeervenVeraarWeiss}, Theorems 1.1 and 1.2]\label{theorem:stochreg}
Let $r,s\in[2,\infty)$ with $s>2$ for $r\neq2$. Then for all $\mathcal{F}$-adapted $H\in L^s((0,\infty)\times \Omega; L^r(G;\cH))$ the stochastic convolution \eqref{eq:stochconvolution} is well defined in $L^r(G)$, $\mathcal{F}$-adapted and the mild solution of \eqref{eq:stoch}.
Moreover, 
\begin{itemize}
\item[i)] for all $\theta \in [0,\frac12)$ there exists a constant $C>0$ such that
\begin{align*}
\E\left[ \|Z\|_{H^{\theta,s}((0,\infty); D(A^{1/2-\theta}))}^s \right] \leq C \E\left[ \|H\|_{L^{s}((0,\infty); L^r(G,\cH))}^s \right],
 \end{align*}
\item[ii)] there exists a constant $C>0$ such that 
\begin{align*}
\E\left[ \|Z\|_{L^{\infty}((0,\infty); D_{A}(1/2,s))}^s \right] \leq C \E\left[ \|H\|_{L^{s}((0,\infty); L^r(G,\cH))}^s \right].
\end{align*}
\end{itemize}
\end{proposition}

For functions $H$ having  better spatial regularity we obtain the following result.

\begin{proposition}[cf. \cite{NeervenVeraarWeiss2}, Theorem 4.5]\label{theorem:stochregimproved}
Under the assumptions of Theorem \ref{theorem:stochreg} assume  additionally that $A^{1/2}H\in L^s((0,\infty)\times \Omega; L^r(G;\cH))$.
\begin{itemize}
\item[i)] For all $\theta \in [0,\frac12)$ there exists $C>0$ such that
\begin{align*}
\E\left[ \|Z\|_{H^{\theta,s}((0,\infty); D(A^{1-\theta}))}^s \right] \leq C \E\left[ \|A^{1/2}H\|_{L^{s}((0,\infty); L^r(G,\cH))}^s \right].
\end{align*}
\item[ii)] There exists $C>0$ such that  
\begin{align*}
\E\left[ \|Z\|_{L^{\infty}((0,\infty);D_{A}(1,s))}^s \right] \leq C \E\left[ \|A^{1/2}H\|_{L^{s}((0,\infty); L^r(G,\cH))}^s \right]. 
\end{align*}
\end{itemize}
\end{proposition}
By Proposition \ref{bidomopthm} and \eqref{eq:interpolation}
\begin{align*}
D(A^{1/2-\theta}) = H^{1-2\theta,r}_N(G)\times L^r(G) \mbox{ and } D_{A}(1/2,s)=B^{1-2/s}_{r,s,N}(G)\times L^s(G).
\end{align*}
Hence, Proposition \ref{theorem:stochreg} and \ref{theorem:stochregimproved} combined with \eqref{eq:interpolation} implies the following regularity result for $Z$.

\begin{corollary}\label{theorem:stochlinex}
Let $s,r\geq 2$, $r>2$ if $s\neq 2$, $H=(h,0)^T$ with $h\in  L^s(\Omega;L^s((0,\infty); L^r(G)))$ $\mathcal{F}$-adapted and $\theta \in [0,1/2)$. 	
	\begin{enumerate}[i)]
		\item Then
			\begin{align*}
			Z &\in L^s \big(\Omega;H^{\theta,s}\big( (0,\infty); H^{1-2\theta,r}_N(G)\times L^r(G)\big)\big)\cap L^s \big(\Omega;L^{\infty}\big( (0,\infty); B^{1-2/s}_{r,s,N}(G)\times L^s(G)\big)\big).
			\end{align*}
			\item If in addition $\aA^{1/2}h\in  L^s(\Omega;L^s((0,\infty); L^r(G)))$, then 
				\begin{align*}
				Z &\in L^s \big(\Omega;H^{\theta,s}\big( (0,\infty); H^{2-2\theta,r}_N(G)\times L^r(G)\big)\big) \cap L^s \big(\Omega;L^{\infty}\big( (0,\infty); B^{2-2/s}_{r,s,N}(G)\times L^s(G)\big)\big).
				\end{align*}
	\end{enumerate}
\end{corollary}

\section{Local existence for various settings}
Given the solution $Z=(z,\zeta)^T$ to \eqref{eq:stoch} and $U$ to \eqref{eq:bidomainstochaddnoise}, the pathwise deterministic system for the remainder $V:=U-Z$ with $V=(v,w)^T$ reads as 
\begin{align}\label{eq:vA}
\begin{split}
\partial_t v + \aA v & = -v^3-3v^2z-3vz^2-z^3+(a+1)(v^2+2vz+z^2)-av-w,\\
\partial_t w &= -bw+cv
\end{split}
\end{align}
with initial conditions ~\eqref{BDESini}, and hence using \eqref{sl-ev-eq} $V$ solves the pathwise deterministic semilinear evolution equation
\begin{align}\label{eq:v}
\partial_t V + AV = F(V+Z),\quad t>0,\quad V(0)=(v_0, w_0)^T.
\end{align}
In the following we establish well-posedness results for \eqref{eq:v} in weak and strong settings. 

\begin{remark}\label{rem:nonlin}
	In order to apply Proposition~\ref{theorem:localexistencedeterministic}, note that the non-autonomous functions $F_1,F_2$ are derived by a shift from autonomous polynomial functions $G_1, G_2$ satisfying Assumption (A3) 
	\begin{align*}
	F_1(t,V):=G_1(V+Z(t)) \mbox{ and } F_2(t,V)=G_2(V+Z(t)). 
	\end{align*}
	Then, for $F_1, F_2$ to fulfill Assumption (A3) it  is sufficient to have $Z\in L^{s}((0,T);X_{\beta})\cap L^{p}((0,T);X_{\mu,p})$ for all $t\in[0,T]$ for some sufficiently large $s$ since $G_1, G_2$ are polynomials and 
	\begin{align*}
	\|G_1(V+Z)\|_{X_0} &\leq C (\|V\|_{X_{\mu,p}}+  \|Z\|_{X_{\mu,p}}), \\
	\|G_2(V +Z)\|_{X_0} &\leq C \sum_{j=1}^{m} \left(1+\|V\|^{\rho_j}_{X_{\beta}}+\|Z\|^{\rho_j}_{X_{\beta}} \right)(\|V\|_{X_{\beta_j}}+\|Z\|_{X_{\beta_j}}).
 	\end{align*}
\end{remark}

\subsection{Strong setting}\label{subsec:strong}\mbox{} \\
\vspace{.1cm}\noindent
In order to apply Proposition \ref{theorem:localexistencedeterministic} 
in the strong setting, i.e.,
\begin{align*}
X_0^s= L^{q}(G)\times L^q(G)\quad \hbox{and} \quad X_1^s= \cD(\aA)\times L^q(G)
\end{align*}
let $1<p,q<\infty$ and  set $\beta=\frac{d}{3q}$. Then 
\begin{align*}
H^{2\beta,q}(G)\hookrightarrow L^{3q}(G),\quad
X_{\beta}^s=H^{2\beta,q}_N(G) \times L^q(G) \mbox{ and } X_{\mu,p}^s= B^{2\mu-2/p}_{q,p,N}(G)\times L^q(G), \quad \mu\in (1/p,1].
\end{align*}
The factor $3$ is due to the order of the polynomials in the FitzHugh-Nagumo nonlinearity which also implies $\rho_1=2$. Now, for $1<q<d$ the minimal choice for $\mu$ to guarantee that assumption (A3) subject to  $m=1$, $\rho_1=2$ and $\beta_1=\beta$ holds is the critical weight $\mu_c$ given by 
$$\mu_c=\frac{1}{p}+\frac{d}{2q}-\frac{1}{2},$$
cf. \cite[Theorem 3.1]{HieberPruess}.
Combining this with the condition  $\mu\in (1/p,1]$, 
we see that  $q>d$ is not admissible, 
and 
that for $1<q<d$ we need to require $\frac{1}{p}+\frac{d}{2q}\leq \frac{3}{2}$.

Now let us turn to the regularity of $z$. Under the assumptions of  Corollary \ref{theorem:stochlinex} i), i.e., for $s,r\geq 2$ with  $r>2$ if $s\neq 2$ one has (pathwise)
\begin{align*}
z \in H^{\theta,s}\big( (0,\infty); H^{1-2\theta,r}_N(G)\big) \quad \hbox{while} \quad 
z\in L^{\eta}((0,\infty);H^{2\beta,q}_N(G))
\end{align*}
for $\eta$ sufficiently large, cf. Remark~\ref{rem:nonlin}, is needed  
to apply Proposition \ref{theorem:localexistencedeterministic}. To match both, let 
\begin{align*}
q>2d/3, \quad r\geq q, \quad \hbox{and} \quad s \geq p \hbox{ with } \beta+1/s \leq 1/2.
\end{align*}
Setting $\theta=1/2-\beta$ we obtain $z\in L^{\eta}((0,\infty);H^{2\beta,q}_N(G))$ for all $\eta\in (1,\infty)$.
Furthermore, we need $z\in L^{p}((0,T); B^{2\mu-2/p}_{q,p,N}(G))$ for some $T>0$, cf. Remark~\ref{rem:nonlin}, and restricting   $\mu\in(1/p,1/p+1/2-1/s]$ with $\mu\geq \mu_c$ we obtain $B^{2\mu-2/p}_{q,p,N}(G)\subset B^{1-2/s}_{r,s,N}(G)$, so $z\in L^{\infty}((0,\infty); B^{2\mu-2/p}_{q,p,N}(G))$ and $z\in L^{p}((0,T); B^{2\mu-2/p}_{q,p,N}(G))$ for all $T>0$.

Letting $h$ as in the assumption of Corollary \ref{theorem:stochlinex} ii), we have 
\begin{align*}
z \in H^{\theta,s}\big( (0,\infty); H^{2-2\theta,r}_N(G)\big), \quad \hbox{and} \quad 
z\in L^{\eta}((0,\infty);H^{2\beta,q}_N(G))
\end{align*}
is needed to apply 
Proposition \ref{theorem:localexistencedeterministic}. So, 
consider first the case  $\beta\leq 1/2$.
Let $r\geq q$ and $s\geq p$, then $\theta$ close to $1/2$, yields the needed embedding  for all $\eta\in (1,\infty)$. 
For  $\beta>1/2$ let  $s$ be such that  $\beta+1/s \leq 1$ and choosing  
$\theta=1-\beta$ one obtains $z\in L^{\eta}((0,\infty);H^{2\beta,q}_N(G))$ for all $1<\eta<\infty$. 
Furthermore, in either case for $s \geq p$, $z\in L^{\infty}((0,\infty); B^{2\mu-2/p}_{q,p,N}(G))$.

In order to summarize the above considerations we introduce the following assumption (Stoch) and (S).

\noindent
{\bf Assumption (Stoch):} Let $1<p,q<\infty$, $2\leq s,r<\infty$, $r>2$ if $s>2$ and  $s\geq p$, $r\geq q$. 

\vspace{.1cm}
\noindent
{\bf Assumption (S):}  Assume (Stoch), $\frac{1}{p}+\frac{d}{2q}\leq \frac{3}{2}$ and let $\mu\in (1/p,1]$ such that  $\mu \geq \frac{1}{p}+\frac{d}{2q}-\frac{1}{2}$. Suppose 
\begin{itemize}
\item[i)] $q>2d/3$, $d/(3q)+1/s \leq 1/2$, $\mu\leq 1/p+1/2-1/s$ if $h\in  L^s(\Omega;L^s((0,\infty); L^r(G)))$;
\item[ii)] $d/(3q)+1/s \leq 1$ if $\aA^{1/2}h\in  L^s(\Omega;L^s((0,\infty); L^r(G)))$. 
\end{itemize}

\begin{proposition}\label{theorem:localpathwise}
Assume (S) and let $Z=(z,\zeta)$ be the solution to \eqref{eq:stoch} given in Corollary~\ref{theorem:stochlinex}. 
Then for 
\begin{align*}
(v_0,w_0)\in B^{2\mu-2/p}_{q,p,N}(G)\times L^q(G),
\end{align*}
there exists $T_0>0$ and a unique solution $(v,w)$ to \eqref{eq:v} with 
\begin{align*}
v \in H^{1,p}_\mu ((0,T_0); L^q(G)) \cap L^p_{\mu}((0,T_0); H^{2,q}_N(G)) , \quad w \in H^{1,p} ((0,T_0); L^q(G)). 
\end{align*}
\end{proposition}
\begin{proof}
 Using Remark~\ref{rem:nonlin} and Assumption (S) together with the above discussed embeddings, the local existence in time-weighted spaces follows from  Proposition \ref{theorem:localexistencedeterministic}. 
 In particular it gives $w \in H^{1,p}_\mu ((0,T_0); L^q(G))$, but the equation for $w$ then yields since 
 $u\in BUC((0,T_0);X_{\mu,p}^s)\subset L^p ((0,T_0); L^q(G))$, see Remark~\ref{rem:embedding}, even
$w \in H^{1,p} ((0,T_0); L^q(G))$. 
\end{proof}

\subsection{Weak setting I}\label{subsec:weakI} \mbox{} \\
	In order to treat also the case $q \ge d$ which is excluded in Assumption (S), we consider the weak setting which, since  $H^{-1,q}(G)=\cD(\aA^{-1/2})$ and
	$H^{1,q}(G)=\cD(\aA^{1/2})$, is given 
	by
	\begin{align*}
	X_0^{\sf w_1} = H^{-1,q}(G) \times L^q(G) \quad \hbox{and}\quad X_1^{\sf w_1} =  H^{1,q}(G) \times L^q(G).  
	\end{align*}	
 
\vspace{.1cm}\noindent
{\bf Assumption (W1):}
Assume (Stoch), $d/(d-1)<q\leq 2d$ with $\frac{1}{p}+\frac{d}{2q}\leq 1$ and $\mu\in (1/p,1]$ such that 
$\mu \geq \frac{1}{p}+\frac{d}{2q}$. Suppose 
\begin{itemize}
\item[(i)] $d/(3q)+1/s \leq 2/3$ if $h\in  L^s(\Omega;L^s((0,\infty); L^r(G)))$; 
\item[(ii)] no further conditions if $\aA^{1/2}h\in  L^s(\Omega;L^s((0,\infty); L^r(G)))$. 
\end{itemize}

\begin{proposition}\label{theorem:localpathwiseweakI}
Assume (W1) and let $Z=(z,\zeta)$ be the solution to \eqref{eq:stoch} given in Corollary \ref{theorem:stochlinex}. 
Then for $$(v_0,w_0)\in B^{2\mu-2/p-1}_{q,p}(G)\times L^q(G),$$ there exists $T_0>0$ and a unique solution $V=(v,w)^T$ to \eqref{eq:v} satisfying
\begin{align*}
v \in H^{1,p}_\mu ((0,T_0); H^{-1,q}(G)) \cap L^p_{\mu}((0,T_0); H^{1,q}(G)), \quad w \in H^{1,p} ((0,T_0); L^q(G)).
\end{align*}
\end{proposition}
\begin{proof}
 For $1<p<\infty$ we have
	\begin{align*}
	X_{\beta}^{\sf w_1}=H^{2\beta-1,q}(G) \times L^q(G) \mbox{ and } X_{\mu,p}^{\sf w_1}= B^{2\mu-2/p-1}_{q,p}(G)\times L^q(G) \quad \hbox{for } \mu\in (1/p,1].
	\end{align*}
	In this situation, we set 
	\begin{align*}
	\beta=\frac{d}{3q}+\frac{1}{3} \quad \hbox{and hence} \quad  \mu_c=\frac{1}{p}+\frac{d}{2q}
	\end{align*}
	which is admissible for all  $q$ in the specified range, compare also \cite[Theorem 3.2]{HieberPruess}.
	
Moreover, since $r\geq q$ and $s\geq p$
 \begin{align}\label{eq:weakI}
 z\in L^{\eta}((0,\infty);H^{2\beta-1,q}_N(G))\cap L^{\infty}((0,\infty); B^{2\mu-2/p-1}_{q,p,N}(G))\quad  \hbox{for $\eta$ large}.
 \end{align}
As in Subsection~\ref{subsec:strong} for the case $\beta=1/2$ we choose $\theta$ close to $1/2$ to obtain this 
	assertion; if $\beta>1/2$, we set $\theta=1-\beta$. In the situation of Corollary~\ref{theorem:stochlinex} ii), $2-2\theta\geq 2\beta-1$ is required, which is fulfilled for all $\theta\in [0,1/2)$. 
	For $r\geq q$ and $s\geq p$ we obtain \eqref{eq:weakI} for any $\eta\in (1,\infty)$ 
	 by taking $\theta$ arbitrary close to $1/2$. With these preparations at hand, the proof can be completed analogously to the one of Proposition~\ref{theorem:localpathwise}.
\end{proof}

\subsection{Weak setting II}\label{subsec:weakII}\mbox{}\\
Aiming to allow $q=2$ in the two dimensional case which is excluded in Assumption (W1), we consider the intermediate case
\begin{align*}
X_0^{\sf w_2}= H^{-1/2,q}(G) \times L^q(G) \quad \hbox{and}\quad X_1^{\sf w_2} =  H^{3/2,q}_N(G) \times L^q(G), 
\end{align*}
where $H^{-1/2,q}(G)=\cD(\aA^{-1/4})$ and $H^{3/2,q}_N(G)=\cD(\aA^{3/4})$. 
In this situation we have for  $1<p<\infty$
\begin{align*}
X_{\beta}^{\sf w_2}=H^{2\beta-1/2,q}_N(G) \times L^q(G) \mbox{ and } X_{\mu,p}^{\sf w_2}= B^{2\mu-2/p-1/2}_{q,p,N}(G)\times L^q(G) \quad \hbox{for } \mu\in (1/p,1].
\end{align*}

\vspace{.1cm}\noindent
{\bf Assumption (W2):}
Assume (Stoch), $2d/(2d-1)<q \leq 4d$ with $\frac{1}{p}+\frac{d}{2q}\leq \frac54$ and let $\mu\in (1/p,1]$ such that 
$\mu \geq \frac{1}{p}+\frac{d}{2q}-\frac{1}{4}$. Assume that one of the following conditions holds.
\begin{itemize}
\item[(i)] $d/(3q)+1/s \leq 7/12$, $h\in  L^s(\Omega;L^s((0,\infty); L^r(G)))$ and $\mu \leq 1/p + 3/4- 1/s$. 
\item[(ii)] $d/(3q)+1/s \leq 13/12$ and  $\aA^{1/2}h\in  L^s(\Omega;L^s((0,\infty); L^r(G)))$. 
\end{itemize}

\begin{proposition}\label{theorem:localpathwiseweakII}
Assume (W2) and let $Z=(z,\zeta)$ be the solution to \eqref{eq:stoch} given in Corollary \ref{theorem:stochlinex}. 
Then for $$(v_0,w_0)\in B^{2\mu-2/p-1/2}_{q,p,N}(G)\times L^q(G),$$ there exists $T_0>0$ and a unique solution $V=(v,w)^T$ to \eqref{eq:v} satisfying
\begin{align*}
v \in H^{1,p}_\mu ((0,T_0); H^{-1/2,q}(G)) \cap L^p_{\mu}((0,T_0); H^{3/2,q}_N(G)), \quad w \in H^{1,p} ((0,T_0); L^q(G)).
\end{align*}
\end{proposition}
\begin{proof}
	We take
	\begin{align*}
	\beta=\frac{d}{3q}+\frac{1}{6} \quad \hbox{for } 2d/(2d-1)<q \leq 4, \hbox{ and } \mu_c=\frac{1}{p}+\frac{d}{2q}-\frac{1}{4}.
	\end{align*}
Hence,  $\frac{1}{p}+\frac{d}{2q}\leq \frac54$. To assure $z\in L^{\eta}((0,\infty);H^{2\beta-1/2,q}_N(G))$ for $\eta$ large we take $r\geq q$ and $s\geq p$. 
	If $\beta\leq 1/4$ we choose  $\theta$ close to $1/2$; if $\beta>1/4$ let $q$ so that $\beta< 3/4$, i.e., $q>4d/7$ and $s$ be large enough to have $\beta+1/s \leq 3/4$ and  take $\theta=3/4-\beta$.
	Furthermore, we assure $z\in L^{\infty}((0,\infty); B^{2\mu-2/p-1/2}_{q,p,N}(G))$ by choosing  $\mu\in(1/p,1/p+3/4-1/s]$ with $\mu\geq \mu_c$.
	
	Assuming $h$ as in the assumption of Corollary \ref{theorem:stochlinex}ii), we need $2-2\theta\geq 2\beta-1/2$ and $\theta s \geq 1$ for some $\theta\in [0,1/2)$. To assure this, we take $r\geq q$, 
	$s\geq p$ and choose $\theta$ similarly as above. 
\end{proof}	

\begin{corollary}\label{cor:bidomainsol}
	Let $(u,w):=(v+z,w)$, where $V=(v,w)^T$ is  given, depending on the setting,  by Proposition  
	\ref{theorem:localpathwise}, \ref{theorem:localpathwiseweakI} or \ref{theorem:localpathwiseweakII}
  for $(v_0,w_0)$ as in Proposition  \ref{theorem:localpathwise}, \ref{theorem:localpathwiseweakI} or \ref{theorem:localpathwiseweakII}, respectively,
	 and $Z=(z,\zeta)$ as in Corollary \ref{theorem:stochlinex}.
	Then
	$U=(u,w)$ is the unique, local solution to the 
	stochastic bidomain problem \eqref{eq:bidomainstochaddnoise} subject ot boundary conditions \eqref{eq:bc} and  initial conditions \eqref{BDESini}
	 satisfying the regularity properties stated in 
	Proposition  \ref{theorem:localpathwise}, \ref{theorem:localpathwiseweakI} or \ref{theorem:localpathwiseweakII}.
\end{corollary}

\begin{remark}
We note that in all settings the critical space is, similarly to the situation of the Navier-Stokes equations, given by $B^{d/q-1}_{q,p,N}(G)\times L^q(G)$. 
\end{remark}

\subsection{From weak I to weak II to strong} \mbox{} \\
Starting with initial values in the weak setting but with a function $h$ that fulfills the requirements of the strong setting, we will use now parabolic regularisation to improve the regularity of the 
solution given in Propositions \ref{theorem:localpathwiseweakI} and \ref{theorem:localpathwiseweakII} to the regularity stated in Proposition \ref{theorem:localpathwise}. 

\begin{lemma}\label{theorem:localsmoothing}
	\begin{itemize}
		\item[a)] Assume (W2) and in addition $p>4/3$ and  $1/p+d/(2q) \leq 1$.  Then the local solution $V=(v,w)$ obtained in Proposition \ref{theorem:localpathwiseweakI} satisfies  for any $0<\delta<T_0$
		\begin{align*}
			v \in H^{1,p} ((\delta,T_0); H^{-1/2,q}(G)) \cap L^p((\delta,T_0); H^{3/2,q}_N(G)) \cap C([\delta,T_0]; B^{3/2-2/p}_{q,p,N}(G)).
		\end{align*}
		\item[b)] Assume (S) and in addition $p>4/3$ and  $1/p+d/(2q) \leq 5/4$.  Then the local solution $V=(v,w)$ obtained in Proposition \ref{theorem:localpathwiseweakII} satisfies for any $0<\delta<T_0$
		\begin{align*}
			v \in H^{1,p} ((\delta,T_0); L^q(G)) \cap L^p((\delta,T_0); H^{2,q}_N(G)) \cap C([\delta,T_0]; B^{2-2/p}_{q,p,N}(G)).
		\end{align*}
		\item[c)] Assume (S) and in addition $p>4/3$ and $1/p+d/(2q) \leq 1$. Then the local solution $V=(v,w)$ obtained in Proposition \ref{theorem:localpathwiseweakI} satisfies for any $0<\delta<T_0$
		\begin{align*}
			v \in H^{1,p} ((\delta,T_0); L^q(G)) \cap L^p((\delta,T_0); H^{2,q}_N(G)) \cap C([\delta,T_0]; B^{2-2/p}_{q,p,N}(G)).
		\end{align*}
	\end{itemize}
\end{lemma}
\begin{proof}
	Let $v$ be solution in the weak-I-setting. Then  $v(t)\in B^{1-2/p}_{q,p,N}(G)$ and we may use $v(t)$ as an  initial value within  the weak-II-setting 
	(Proposition \ref{theorem:localpathwiseweakII}) provided $\mu\leq 3/4$ is admissible, i.e., provided  $p>4/3$ and provided  the critical weight $\mu_c$ in the weak-II-setting satifies 
	$\mu_c \leq 3/4$. \\
	Furthermore, let $v$ be a solution in the weak-II-setting. Then $v(t)\in B^{3/2-2/p}_{q,p,N}(G)$ and we may  use $v(t)$ as an initial value in the strong setting 
	(Proposition \ref{theorem:localpathwise}) provided  $\mu\leq 3/4$ is admissible, i.e., provided $p>4/3$ and provided  the critical weight $\mu_c$ in the strong setting satisfies $\mu_c \leq 3/4$.
\end{proof}

\section{Global Existence for $d=2$ and $d=3$} \label{sec:global}

\subsection{Global existence for $d=2$} \mbox{} \\

\noindent
Our main result in the two dimensional setting reads as follows.

\begin{theorem}\label{thm:globald2}
	Let $d=2$, assume  (BD), (S) and in addition $p \in[2,\infty)$ and $q \in [2,4]$ with  $\frac{1}{p}+\frac{1}{q}\geq \frac{1}{2}$. 
	Then for any $T>0$ and for each $$(v_0,w_0)\in B^{2/q-1}_{q,p,N}(G) \times L^q(G),$$ equation \eqref{eq:v} admits  a unique strong solution $V=(v,w)^T$ within the regularity class
	\begin{align*}
	v &\in H^{1,p}_\mu ((0,T); H^{-1,q}(G)) \cap L^p_{\mu}((0,T); H^{1,q}(G))\cap C([0,T]; B^{2/q-1}_{q,p,N}(G)), \\
	w &\in H^{1,p} ((0,T); L^q(G)),
	\end{align*}
	and for any $\delta\in (0,T)$
	\begin{align*}
	v \in H^{1,p} ((\delta,T); L^q(G)) \cap L^p((\delta,T); H^{2,q}_N(G))\cap C([\delta,T]; B^{2-2/p}_{q,p,N}(G)).  
	\end{align*}
	The function $U=V+Z$, where $Z=(z,\zeta)^T$ is given by \eqref{eq:stochconvolution},  is the unique, global, pathwise solution to the stochastic bidomain equations \eqref{eq:bidomainstochaddnoise} with $V$ in the regularity class given above.	
\end{theorem}

\begin{remark}
Note that $L^2(G)\hookrightarrow B^{2/q-1}_{qp}(G)$ for all $p,q\geq2$, and that  hence in particular initial values $(v_0,w_0) \in L^2(G) \times L^2(G)$ are covered by Theorem \ref{thm:globald2}. 
The combination of stochastic maximal regularity with the deterministic theory of critical spaces shows thus  that the stochastic bidomain equations \eqref{BDES} are not only globally well posed for 
these data in the weak sense as shown  by Bendahmane and Karlsen in \cite{BK19}, 
but even in the {\em strong} sense.  
\end{remark}

The proof of Theorem \ref{thm:globald2} is based on two facts: knowing from Proposition \ref{bidomopthm} b) that $D(\aA^{1/2})=H^{1,q}(G)$, it is possible to derive energy bounds for the 
solution of \eqref{eq:v}, and secondly these bounds can be related within the weak-II-setting  to the blow up criteria in critical spaces given in Proposition \ref{theorem:serrin}. 

The following lemma gives a priori bounds for $v$ and $w$ independent of the space dimension and holds hence for $d=2$ and $d=3$. 

\begin{lemma}[\textit{A priori} bound on $v$ and $w$]\label{lem:eb}
Let $q\geq 2$, $\delta>0$ and let $V=(v,w)$ be the solution to  
\eqref{eq:v}
obtained in one of the Propositions  \ref{theorem:localpathwise}, \ref{theorem:localpathwiseweakI} 
or \ref{theorem:localpathwiseweakII} on some time interval $(0,T)$. Then 
\begin{align*}
v  \in L^{\infty}((\delta,T);L^2(G)) \cap L^2((\delta,T);H^{1,2}(G))\cap L^4((\delta,T);L^4(G)), \quad w  \in L^{\infty}((\delta,T);L^2(G)),
\end{align*}
and there exists a constant $C>0$ such that 
\begin{multline*}
\|v(t)\|_{L^2(G)}^2+\|w(t)\|_{L^2(G)}^2 +\int_{\delta}^{t}\|\aA^{1/2}v(s)\|_{L^2(G)}^2\ds+\int_{\delta}^{t}\|v(s)\|_{L^4(G)}^4\ds\\
\leq C \big(\|v(\delta)\|_{L^2(G)}^2+\|w(\delta)\|_{L^2(G)}^2+\int_{\delta}^{T}\|z(s)\|_{L^2(G)}^2+\|z(s)\|_{L^4(G)}^4 \ds \big)e^{cT}, \quad 0<\delta < t < T.
\end{multline*}
\end{lemma}

\begin{proof}
Multiplying equation \eqref{eq:v} by $v$ in $L^2(G)$ yields
\begin{align*}
 \frac12 \partial_t &\|v\|_{L^2(G)}^2 + \left<\aA v,v\right> +\|v\|_{L^4(G)}^4+ a\|v\|_{L^2(G)}^2+3\|vz\|_{L^2(G)}^2\\
& \leq \|v^3z\|_{L^1(G)} + \|z^3v\|_{L^1(G)}
+\|wv\|_{L^1(G)} + 2(\|v^3\|_{L^1(G)}+2\|v^2z\|_{L^1(G)}+\|vz^2\|_{L^1(G)})\\
&\leq \|v^3\|_{L^{\frac43}(G)}\|z\|_{L^4(G)} \! +\! \|z^3\|_{L^{\frac43}(G)}\|v\|_{L^4(G)} 
+ \! \|w\|_{L^{2}(G)}\|v\|_{L^2(G)}\\
& \quad + 2 (\|v^2\|_{L^{2}(G)}\|v\|_{L^2(G)}+2\|v^2\|_{L^{2}(G)}\|z\|_{L^2(G)}+\|v\|_{L^{2}(G)}\|z^2\|_{L^2(G)})\\
&\leq \|v\|_{L^{4}(G)}^3\|z\|_{L^4(G)} + \|z\|_{L^{4}(G)}^3\|v\|_{L^4(G)}
+\|w\|_{L^{2}(G)}\|v\|_{L^2(G)}\\
&\quad + 2(\|v\|_{L^{4}(G)}^2\|v\|_{L^2(G)}+2\|v\|_{L^{4}(G)}^2\|z\|_{L^2(G)}+\|v\|_{L^{2}(G)}\|z\|_{L^4(G)}^2)\\
&\leq \frac12 \|v\|_{L^{4}(G)}^4+c\|z\|_{L^4(G)}^4 + c\|v\|_{L^2(G)}^2+c\|z\|_{L^2(G)}^2+\frac{b}{4}\|w\|_{L^{2}(G)}^2.
\end{align*}
The equation for $w$ gives
\begin{align*}
\frac12 \partial_t \|w\|_{L^2(G)}^2+b\|w\|_{L^2(G)}^2 \leq \frac{b}{4} \|w\|_{L^{2}(G)}^2+c\|z\|_{L^2(G)}^2 + c\|v\|_{L^2(G)}^2.
\end{align*}
Using the fact that $\cD(\aA^{1/2})=H^{1,q}(G)$, see Proposition \ref{bidomopthm} and \eqref{eq:interpolation}, we obtain $$\left<\aA v,v\right> + \norm{v}^2= \|\aA^{1/2}v\|_{L^2(G)}^2 + \norm{v}^2 \geq C \|v\|_{H^{1,2}(G)}^2,$$ and 
Gronwall's inequality yields  the assertion.
\end{proof}

\begin{remark}\label{nohoped3}
It is interesting to consider the Sobolev indices of the above bounds for $v$. The first two terms have index $-d/2$, the third one has index $-(1/2+d/4)$. For $d=3$ both indices are strictly less than 
$-1$, the Sobolev index of critical spaces. Hence, for $d=3$ there is no hope to prove global existence results for $v$ based on these elementary energy estimates. However, for $d=2$, the energy bounds relate
via interpolation to the critical space $X_{\mu_c}$ given by $X_{\mu_c}=H^{2/p+2/q-1,q}(G)\times L^q(G)$. 
\end{remark}

\begin{proof}[Proof of Theorem~\ref{thm:globald2}]
Due to $p,q\geq 2$ and $\frac{1}{p}+\frac{1}{q}\leq \frac{1}{2}$, we have the embeddings 
\begin{align*}
L^{\infty}((\delta,T); L^2(G))\cap L^2((\delta,T);H^{1,2}(G)) \hookrightarrow L^p((\delta,T); H^{2/p,2}(G)) \hookrightarrow L^p((\delta,T); H^{2/p+2/q-1,q}(G)).
\end{align*}
The energy bounds imply $w \in L^p((\delta,T); L^q(G))$.  Here, $$X_{\mu_c}=H^{2/p+2/q-1,q}(G)\times L^q(G),$$  and hence  the global existence  
follows from Proposition  \ref{theorem:serrin}. 
\end{proof}

\subsection{Global existence for $d=3$}\mbox{}\\

\noindent
We now state our main result in the three dimensional setting.  

\begin{theorem}\label{thm:globd3}
Let $d=3$, assume (BD), (S) and in addition $p\in [2,\infty)$ and $ q \in [2, 6]$. Moreover, let $r\geq 6$ in case (i) of Assumption (S). 
Then for any $T>0$ and for each 
\begin{align*}
 v_0\in B^{3/q-1}_{q,p,N}(G) \quad \hbox{and}\quad  w_0\in H^{1,2}(G) \cap L^q(G)
\end{align*}
equation \eqref{eq:v} admits  a unique strong solution $V=(v,w)^T$ within the regularity class
\begin{align*}
v &\in H^{1,p}_\mu ((0,T); H^{-1,q}(G)) \cap L^p_{\mu}((0,T); H^{1,q}_N(G)) \cap C([0,T]; B^{3/q-1}_{q,p,N}(G)),
\\
w &\in H^{1,p} ((0,T); L^q(G)),
\end{align*}
and for $\delta\in (0,T)$
\begin{align*}
v \in H^{1,p} ((\delta,T); L^q(G)) \cap L^p((\delta,T); H^{2,q}_N(G)) \cap C([\delta,T]; B^{2-2/p}_{q,p,N}(G)). 
\end{align*}
The function $U=V+Z$, where $Z=(z,\zeta)^T$ is given by \eqref{eq:stochconvolution},  is the unique, global pathwise solution to  the stochastic bidomain equations \eqref{eq:bidomainstochaddnoise} with $V$ in the regularity class given above.
\end{theorem}

\begin{remark}
Observe that $H^{1/2,2}(G) \hookrightarrow B^{3/q-1}_{qp}(G)$ for all $p,q \geq 2$ and that thus initial data $(v_0,w_0) \in H^{1/2,2}(G) \times H^{1,2}(G)$ are covered by Theorem \ref{thm:globd3}. 
\end{remark}

Following Remark \ref{nohoped3}, there is no hope to prove Theorem \ref{thm:globd3} by purely applying the energy estimates given in Lemma \ref{lem:eb}. In the deterministic case, one applies parabolic 
regularization to differentiate the equation and to apply energy estimates for $v'$ and $w'$ to show that $\lim_{t \to T_{\max}} v(t)$ exists in $B^{3/q-1}_{q,p}(G)$ and that thus the solution exists 
globally. In the stochastic case, we cannot differentiate the equation and estimate instead the term $\|\nabla \partial_t v\|_{L^2(G)}$. However, in doing so we need to assume a better regularity of the 
initial value $w_0$, since there is no spatial smoothing for $w$.

\begin{lemma}[\textit{A priori} bound on  $w$]
Let $w_0\in H^{1,2}(G)$. Then there exists $C=C(T)>0$ such that
\begin{align*}
\norm{w}_{L^{\infty}((0,T);H^{1,2}(G))}\leq C.
\end{align*}
\end{lemma}

\begin{proof}
Applying $\aA^{1/2}$ to the equation for $w$ and multiplying by $\aA^{1/2} w$ we obtain
\begin{align*}
\frac12 \partial_t \|\aA^{1/2} w\|_{L^2(G)}^2+\frac{b}{2}\|\aA^{1/2}w\|_{L^2(G)}^2
& \leq c\|\aA^{1/2} z\|_{L^2(G)}^2 +c < \aA^{1/2} v,\aA^{1/2} w>\\
& = c\|\aA^{1/2} z\|_{L^2(G)}^2 +c < t^{(1-\mu)/2}\aA^{1/2} v,t^{(\mu-1)/2}\aA^{1/2} w>
\end{align*}
and thus
\begin{align*}
\partial_t \|\aA^{1/2} w\|_{L^2(G)}^2 +\|\aA^{1/2}w\|_{L^2(G)}^2 \leq c\|\aA^{1/2} z\|_{L^2(G)}^2 +t^{1-\mu} c\|\aA^{1/2}v\|_{L^2(G)}^2+t^{\mu-1}\|\aA^{1/2}w\|_{L^2(G)}^2.
\end{align*}
Hence, 
\begin{align*}
\|\aA^{1/2} w \|_{L^2(G)}^2  \leq c\Big(\|\aA^{1/2} w(0)\|_{L^2(G)}^2+\|\aA^{1/2} z\|_{L^2((0,T);L^2(G)}^2 +\|\aA^{1/2}v\|_{L^2_{\mu}((0,T);L^2(G)}^2\Big)e^{-bT+1/\mu T^{\mu}},
\end{align*}
and the norm of $w$ in $L^{\infty}((0,T);H^{1,2}(G))$ is uniformly bounded.
\end{proof}

We now show in  two steps how to obtain additional regularity for $v$ due to parabolic smoothing. Let us start with an estimate for 
$v\in H^{1,2}((\delta,T); L^2(G))\cap L^{\infty}((\delta,T);H^{1,2}(G))$.

\begin{lemma}[Second \textit{a priori} bound on $v$]\label{lemma:d3_1}
Assume  $r\geq 6$ if case (i) of Assumption (S) applies. Then for any $0<\delta<T$ 
\begin{align*} 
v\in H^{1,2}((\delta,T); L^2(G))\cap L^{\infty}((\delta,T);H^{1,2}(G))
\end{align*}
with
\begin{multline*}
2\|\aA^{1/2}v(t)\|^2_{L^2(G)}+  \|v(t)\|_{L^4(G)}^4+2\int_\delta^t\|\partial_t v(t)\|_{L^2(G)}^2\d s 
 \leq \big(2\|\aA^{1/2}v(\delta)\|^2_{L^2(G)}+ \|v(\delta)\|_{L^4(G)}^4+ \\
  c\int_\delta^T 
\|z\|_{L^{4}(G)}^4+\|z\|_{L^{6}(G)}^6+\|z\|_{L^{8}(G)}^4+\|v\|_{L^{2}(G)}^2+\|w\|_{L^2(G)}^2 \ds\big)
 \cdot e^{cT+c\int_\delta^T\|z\|_{L^{\infty}(G)}^2+\|z\|_{L^{8}(G)}^4\d s}.
\end{multline*}
\end{lemma}

\begin{proof}
Multiplication of \eqref{eq:vA} by $\partial_t v$ and parabolic regularisation yield
\begin{align*}
&\|\partial_t v\|_{L^2(G)}^2 +\frac{1}{2}\partial_t \|\aA^{1/2}v\|^2_{L^2(G)}+\frac{1}{4} \partial_t \|v\|_{L^4(G)}^4 \\
&\qquad =<-3v^2z-3vz^2-z^3+(a+1)(v^2+2vz+z^2)-av-w,\partial_t v>\\
&\qquad \leq \frac{1}{2}\|v_t\|_{L^2(G)}^2+c\|3v^2z-3vz^2-z^3+(a+1)(v^2+2vz+z^2)-av-w\|_{L^2(G)}^2\\
&\qquad \leq \frac{1}{2}\|v_t\|_{L^2(G)}^2+c\left(\|z\|_{L^{\infty}(G)}^2\|v\|_{L^4(G)}^4+\|z\|_{L^{8}(G)}^4\|v\|_{L^4(G)}^2\right. \\
&\qquad\qquad\qquad\qquad\qquad\quad  \left. +\|z\|_{L^{6}(G)}^6
+\|v\|_{L^4(G)}^4+\|z\|_{L^{4}(G)}^4+\|v\|_{L^{2}(G)}^2+\|w\|_{L^2(G)}^2\right).
\end{align*}
We have $z \in L^2((0,T);L^{\infty}(G))$ in case (i) by choosing $\theta=0$ in Corollary \ref{theorem:stochlinex}, $z \in L^4((0,T);L^{8}(G))$ follows by choosing $\theta=\frac14$, while $z \in L^6((0,T);L^{6}(G))$ is obtained in case i) by the choice $\theta=1/3$.
Hence, we by Gronwall's inequality the stated estimate follows.
\end{proof}

\noindent
The improved regularity for $v$ gives by Sobolev embedding $v\in L^{\infty}((\delta,T);L^6(G))$, which can be used to estimate the right hand side $$R:= v^3+3v^2z+3vz^2+z^3-(a+1)(v^2+2vz+z^2)+av+w$$  in equation \eqref{eq:vA} and implies an a-priori bound for $v$ in $H^{1,2}((\delta,T);H^{1,2}(G))$.

\begin{lemma}[Third \textit{a priori} bound on $v$]
Let additionally $r\geq 6$ in case (i) of Assumption (S). Then for any $0<\delta<T$ 
\begin{align*}
v\in  L^{\infty}((\delta,T);H^{2,2}(G)) \cap H^{1,2}((\delta,T);H^{1,2}(G)).
\end{align*}
\end{lemma}
\begin{proof}
The Gagliardo-Nirenberg inequality gives $\|\aA^{1/2}v\|_{L^6(G)}\leq c\|\aA v\|_{L^2(G)}+c\|v\|_{L^6(G)}$, and with this we estimate the terms in $\aA^{1/2}R$ by
\begin{align*}
\|\aA^{1/2} (v^3)\|^2_{L^2(G)}
&\leq c\|v^2\|^2_{L^3(G)}\|\aA^{1/2} v\|^2_{L^6(G)}+c\|\aA^{1/2} v^2\|^3_{L^3(G)}\|v\|^2_{L^6(G)}
\leq c\|v\|_{L^6(G)}^4 \|\aA^{1/2}v\|^2_{L^6(G)}\\
&\leq c\|v\|_{L^6(G)}^4 \|\aA v\|^2_{L^2(G)}+c\|v\|_{L^6(G)}^6,\\
\|\aA^{1/2} (v^2 z)\|^2_{L^2(G)}
&\leq c\|v^2\|^2_{L^3(G)}\|\aA^{1/2} z\|^2_{L^6(G)}+c\|\aA^{1/2}v^2\|^2_{L^3(G)}\|z\|^2_{L^6(G)}\\
&\leq c\|v\|_{L^6(G)}^4 \|\aA^{1/2} z\|^2_{L^6(G)}+ c\|v\|_{L^6(G)}^4\|z\|^2_{L^6(G)} +c\|\aA v\|^2_{L^2(G)}\|v\|^2_{L^6(G)}\|z\|^2_{L^6(G)},\\
\|\aA^{1/2} (vz^2)\|^2_{L^2(G)}
&\leq c\|z^2\|^2_{L^3(G)}\|\aA^{1/2} v\|^2_{L^6(G)}+c\|\aA^{1/2} z^2\|^2_{L^3(G)}\|v\|^2_{L^6(G)}\\
&\leq c\|z\|_{L^6(G)}^4\|\aA v\|^2_{L^2(G)}+c\|z\|_{L^6(G)}^4\|v\|_{L^6(G)}^2+c\|\aA^{1/2} z\|^2_{L^6(G)}\|z\|^2_{L^6(G)}\|v\|^2_{L^6(G)},\\
\|\aA^{1/2} (z^3)\|^2_{L^2(G)}
&\leq c\|z\|_{L^6(G)}^4\|\aA^{1/2} z\|^2_{L^6(G)},
\end{align*}
as well as
\begin{align*}
\|\aA^{1/2} (v^2+2vz+z^2)\|^2_{L^2(G)}
& \leq c\Big(\|\aA^{1/2} v\|^2_{L^4(G)}\| v\|^2_{L^4(G)}+\|\aA^{1/2} z\|^2_{L^4(G)}\| v\|^2_{L^4(G)} \\
&\qquad  +\|\aA^{1/2} v\|^2_{L^4(G)}\| z\|^2_{L^4(G)} +\|\aA^{1/2} z\|^2_{L^4(G)}\| z\|^2_{L^4(G)} \Big)\\
& \leq c\Big(\|\aA v\|^2_{L^2(G)}\| v\|^2_{L^4(G)}+\| v\|^4_{L^4(G)}+\|\aA^{1/2} z\|^2_{L^4(G)}\| v\|^2_{L^4(G)} \\
&\qquad  +\|\aA  v\|^2_{L^2(G)}\| z\|^2_{L^4(G)}+\| z\|^2_{L^4(G)} \| v\|^2_{L^4(G)}  +\|\aA^{1/2} z\|^2_{L^4(G)}\| z\|^2_{L^4(G)} \Big).
\end{align*}
Noting that
\begin{align*}
 \|\aA^{1/2} (av+w)\|_{L^2(G)} \leq c(\|\aA^{1/2} v\|_{L^2(G)}  
 +\|\aA^{1/2} w\|_{L^2(G)} )
\end{align*}
we finally obtain
\begin{align*}
\|\aA^{1/2} R\|_{L^2(G)}^2\leq& c(\|v\|_{L^6(G)}^4+\| v\|^2_{L^4(G)}+\| z\|^2_{L^4(G)} +\|z\|_{L^6(G)}^4 ) \|\aA v\|^2_{L^2(G)}\\
 &+c(\|\aA^{1/2} w\|^2_{L^2(G)}  +\|\aA^{1/2} v\|^2_{L^2(G)} +\|v\|_{L^6(G)}^6+\| v\|^4_{L^4(G)}+\| z\|^4_{L^4(G)}+\|z\|^6_{L^6(G)})\\
 &\qquad +c(\|v\|_{L^6(G)}^4 +\|z\|_{L^6(G)}^4)\|\aA^{1/2} z\|^2_{L^6(G)}+c(\|v\|_{L^4(G)}^2 +\|z\|_{L^4(G)}^2)\|\aA^{1/2} z\|^2_{L^4(G)}.
\end{align*}
In case i) of assumption (S) we have $\aA^{1/2}z\in L^2((0,T);L^6(G))$ by choosing $\theta=0$ in Corollary \ref{theorem:stochlinex} and $z\in L^{\infty}((0,T);L^6(G))$ by choosing $\theta\in(1/s,1/2)$. Note, that $s>2$ since $r\neq 2$. So we have
\begin{align*}
c\int_\delta^T \|v\|_{L^6(G)}^4+\| v\|^2_{L^4(G)}+\| z\|^2_{L^4(G)} +\|z\|_{L^6(G)}^4 \d s\leq c
\end{align*}
and
\begin{align*}
&\int_\delta^T c(\|\aA^{1/2} w\|^2_{L^2(G)}  +\|\aA^{1/2} v\|^2_{L^2(G)} +\|v\|_{L^6(G)}^6+\| v\|^4_{L^4(G)}+\| z\|^4_{L^4(G)}+\|z\|^6_{L^6(G)})\\
 &\qquad +c(\|v\|_{L^6(G)}^4 +\|z\|_{L^6(G)}^4)\|\aA^{1/2} z\|^2_{L^6(G)}+c(\|v\|_{L^4(G)}^2 +\|z\|_{L^4(G)}^2)\|\aA^{1/2} z\|^2_{L^4(G)} \ds \leq c
\end{align*}
In case ii) of assumption (S) we can argue in the same way if $s>2$, since $D(\aA^{1/2})\subset L^6(G)$. For $s=2$ and hence also $r=2$ we obtain $\aA^{1/2}z\in L^2((0,T);L^6(G))$ by choosing again $\theta=0$ and because of $z\in L^{\infty}((0,T);B^1_{2,2}(G))= L^{\infty}((0,T);H^{1,2}(G))$ we get $z\in L^{\infty}((0,T);L^6(G))$.\\
Multiplying now \eqref{eq:vA} with $\aA \partial_t v$ we get
\begin{align*}
\partial_t \|\aA v\|_{L^2(G)}^2+ \|\aA^{1/2} \partial_t v\|_{L^2(G)}^2 \leq \|\aA^{1/2} R\|_{L^2(G)}^2,
\end{align*}
and by Gronwall's inequality we obtain $v\in L^{\infty}((\delta,T);H^{2,2}(G)) \cap H^{1,2}((\delta,T);H^{1,2}(G))$. \qedhere
\end{proof}

\begin{proof}[Proof of Theorem~\ref{thm:globd3}]
The existence of a local solution was already established in  Proposition  \ref{theorem:localpathwiseweakII} for $q<6$  and in Proposition \ref{theorem:localpathwiseweakI} for $q=6$. 
By Lemma  \ref{theorem:localsmoothing} this local solution belongs to 
\begin{align*}
	H^{1,p} ((\delta,T); L^q(G)) \cap L^p((\delta,T); H^{2,q}_N(G)) \cap C([\delta,T]; B^{2-2/p}_{q,p,N}(G))
\end{align*}
for all $0<\delta<T$. The embeddings
$$H^{1,2}((\delta,T);H^{1,2}(G)) 
\hookrightarrow BC^{\frac12}([\delta, T); B^{\frac{3}{q}+1-\frac{3}{2}}_{q,p}(G)) 
\hookrightarrow BC^{\frac12}([\delta,T); B^{\frac{3}{q}+1-\frac{3}{2}}_{q,p}(G))
\hookrightarrow BUC ([\delta, T);B^{\frac{3}{q}-1}_{q,p}(G))$$
yield the existence of $\lim_{t\to T_{\max}} v(t)$ in $B^{\frac{3}{q}-1}_{qp}(G)$. Corollary \ref{corollary:continuation} implies then the global existence of $v$ and thus the assertion. 
\end{proof}

\noindent
{\bf Acknowledgement.} The authors would like to thank Zdzislaw Brzezniak for stimulating discussions. The third author  gratefully acknowledges the financial support of the Deutsche Forschungsgemeinschaft (DFG) through the research fellowship SA 3887/1-1.

\end{document}